\documentclass{amsart}
\usepackage{amsfonts,amssymb,amscd,amsmath,enumerate,verbatim,calc}
\usepackage{xy}

\newcommand{\CM}{Cohen-Macaulay}

\newcommand{\wrt}{with respect to}

\newcommand{\n}{\mathfrak{n} }
\newcommand{\m}{\mathfrak{m} }

\newcommand{\rt}{\rightarrow}

\newcommand{\ov}{\overline}

\newcommand{\low}{\ell\ell}

\newcommand{\wh}{\widehat }

\newcommand{\ed}{\operatorname{embdim}}

\newcommand{\height}{\operatorname{height}}

\newcommand{\Ass}{\operatorname{Ass}}
\newcommand{\type}{\operatorname{type}}
\newcommand{\soc}{\operatorname{soc}}

\newcommand{\projdim}{\operatorname{projdim}}
\newcommand{\injdim}{\operatorname{injdim}}

\newcommand{\ord}{\operatorname{ord}}
\newcommand{\Hom}{\operatorname{Hom}}
\newcommand{\Tor}{\operatorname{Tor}}

\theoremstyle{plain}

\newtheorem{theorem}{Theorem}[section]

\newtheorem{proposition}[theorem]{Proposition}

\theoremstyle{definition}

\newtheorem{s}[theorem]{}
\newtheorem{remark}[theorem]{Remark}

\theoremstyle{remark}

\begin{document}

\title[Lowey length]{On the Lowey length of modules of finite projective dimension-II.}
\author{Tony~J.~Puthenpurakal}
\date{\today}
\address{Department of Mathematics, IIT Bombay, Powai, Mumbai 400 076}

\email{tputhen@math.iitb.ac.in}
\subjclass{Primary 13D02; Secondary 13D40 }
\keywords{index, Loewy length}

 \begin{abstract}
Let $(A,\m)$ be a  Gorenstein local ring
of dimension $d \geq 1$.
Suppose there exists
 be a non-zero $A$ module
 $M$ of finite
length and finite projective dimension such that $\low(M)$, the Lowey length of $M$, is
  equal to $\lambda(M)$, the length
of $M$.
Then we show that  necessarily $A$ is at worst a hypersurface singularity. We also
characterize  Gorenstein
local rings having a non-zero module $M$ of finite length and finite projective dimension   with $\low(M) = \lambda(M)-1$.

\end{abstract}
 \maketitle
\section{introduction}
\noindent Dear Reader, \\
While reading this paper
it will be a good idea to keep a copy of part 1, \cite{P}, of this paper with you. The notation used in this paper are
same as that of \cite{P}.

Let $(A,\m)$ be a Gorenstein local ring of postive dimension and let
$M$ be a non-zero module of finite length and finite projective dimension. In part 1 of this paper we showed that the Lowey
length $\low(M) \geq \ord(A)$. This generalizes a result of Avramov, Buchweitz, Iyengar and Miller, \cite[1.1]{ABIM}.
 Thus the singularity of $A$ imposes a lower bound on the Lowey length of modules of finite projective dimension.

 In this paper we investigate upper bounds of Lowey lengths of finite length modules of finite projective dimension.
 An obvious
 upper bound  is
 $\low(M) \leq  \lambda(M)$.
 Note that if $A$ is regular then $k = A/\m$.
 has finite projective dimension and
 $\low(k) = \lambda(k) = 1$. Our first result is
 \begin{theorem}\label{main-1}
 Let $(A,\m)$ be a \CM \ local ring of  dimension $d \geq 1$. Suppose there exists an
$A$-module $M$ of non-zero finite length module of finite projective dimension such that
 $\low(M) = \lambda(M)$.  Then $A$ is at worst a hypersurface singularity, i.e., the completion $\widehat{A} =  Q/(f)$ where
 $(Q,\n)$ is a complete regular local ring and $f \in \n$.
 \end{theorem}
 Note that we do not require
 $A$ to be Gorenstein in Theorem \ref{main-1}.

We complement Theorem \ref{main-1} by proving
\begin{proposition}\label{Ex-1}
 Let $(A,\m)$ be an abstract hypersurface ring of positive dimension. Then there exists a
 finite length $A$- module $M$ of finite projective dimension with $\low(M) = \lambda(M)$.
\end{proposition}

A natural question is what happens if $\lambda(M) - \low(M) = 1$?
Our next result implies that this still forces strong conditions on $A$.
Set $\ed(A)$ to be the embedding dimension of $A$, i.e, the number of minimal generators of $\m$.
We show
\begin{theorem}\label{main-2}
 Let $(A,\m)$ be a Gorenstein local ring of dimension $d \geq 1$
 and let $M$ be a non-zero finite length module of finite projective dimension. If $\lambda(M) - \low(M) = 1$ then
$\ed(A) - d \leq 2$ (in particular $A$ is a complete intersection of codimension $\leq 2$). Furthermore we have
\begin{enumerate}[\rm (1)]
  \item Assume $\ed(A) - d = 1$ (the hypersurface ring case).  If $\wh{A} = Q/(f)$ where $(Q,\n)$ is regular local then $f \in \n^2 \setminus \n^3$.
  \item Assume $\ed(A) - d = 2$. If $\wh{A} = Q/(f, g)$ where $(Q,\n)$ is regular local then $f,g \in \n^2 \setminus \n^3$.  Furthermore the images of $f,g$ in $\n^2/\n^3$ is linearly independent over $k = A/\m$.
  \item If $M$ is not cyclic then necessarily $A$ is regular local.
\end{enumerate}
\end{theorem}

We complement Theorem \ref{main-2} by giving examples \ref{examples-thm-2} which shows that all the cases above occur.

To prove our results we use some elementary techniques which arise in the study of associated graded rings of $\m$-primary ideals
in Cohen-Macaulay rings. It is suprising to us that a single module $M$ imposes such restrictions on the ring $A$.

\section{preliminaries}
In this section we discuss a few preliminaries that we need.

\s \label{complete}
\emph{reduction to the case when $A$ is complete:} \\
Let $\widehat{A}$ be the completion of $A$ \wrt \  $\m$. If $M$ is a finitely generated $A$-module then set $\widehat{M} = M \otimes_A \widehat{A}$
the completion of $M$. Then the following results are well-known:
\begin{enumerate}
 \item $\lambda_A(M) = \lambda_{\widehat{A}}(\widehat{M}).$
 \item
 $\low_A(M) = \low_{\widehat{A}}(\widehat{M})$
 \item
 $\dim M = \dim \widehat{M}$.
 \item
 $\ed(A) = \ed(\widehat{A})$
 \item
 $\ord(A) = \ord(\widehat{A})$.
\item
$\projdim_A M = \projdim_{\widehat{A}} \widehat{M}$.
 \end{enumerate}
We also note that if $M$ has finite length then it has a natural structure as an $\widehat{A}$-module and $\widehat{M} \cong M$
as $A$-modules.

We need the following  well known result:
\begin{proposition} \label{dual}
Let $(A,\m)$ be a complete local ring and let $E$ be the injective hull of $k = A/\m$. Let $M$ be a finite length
 $A$-module. Set $M^\vee = \Hom_A(M, E)$. Then
 \begin{enumerate}[\rm(1)]
  \item $M \cong (M^\vee)^\vee$.
  \item
  $\lambda(M^\vee) = \lambda(M)$.
  \item
  $\low(M^\vee) = \low(M)$.
  \item If $\projdim M < \infty$ then $\injdim M^\vee < \infty$. Furthermore if $A$ is Gorenstein then $\projdim M < \infty$.
 \end{enumerate}
\end{proposition}
\begin{proof}
 The first two assertions are well-known, cf  \cite[3.2.12]{BH}. \\
 (3) Suppose $\m^c M = 0$. Let $f \in M^\vee$, $m \in M$ and let $a \in  \m^c$.
 Then
 $$(af)(m) = f(am) = f(0) = 0$$
 It follows that $\low(M) \leq \low(M^\vee)$.

 By same argument we get $\low(M^\vee) \leq \low( M^{\vee \vee})$. But $M^{\vee \vee} \cong M$. The result follows.
 
 (4) This is well-known.
\end{proof}

\s Let $G(A) = \bigoplus_{n \geq 0} \m^n/\m^{n+1}$ be the associated graded ring of $A$ (\wrt \ to $\m$) and let
$G(M) = \bigoplus_{n \geq 0} \m^n M/ \m^{n+1}M$ be the associated graded module of $M$. Clearly if $M$ is finitely generated
as an $A$-module then $G(M)$ is a finitely generated $G(A)$-module.

\s Let $\beta_i(M) = \lambda(\Tor^A_i(M,k))$ be the $i^{th}$ betti number of $M$. Let $P^A_M(z) = \sum_{i\geq 0} \beta_i(M) z^n$
be the \emph{Poincare series} of $M$.

We will need the well-known result:
\begin{proposition}\label{divide}
 Let $A$ be a \CM \ local ring of dimension $d \geq 1$. Let $N$ be  a finite length module and also assume projective dimension of
 $N$ is finite. Then $(1+z)$ divides $P_N^A(z)$.
\end{proposition}
\begin{proof}
\textit{Sketch:}  Simply  localize a minimal free resolution of $N$ at a minimal prime of $A$.
\end{proof}

The following result is needed later on.

\begin{proposition}\label{lean}
 Let $(A,\m)$ be a Noetherian local ring and let $E$ be a finitely generated $A$-module.
 Assume $\m^i E = < p>$ and $\mu(\m^{i+1} E) \leq 1$. Then there exists $x \in \m$ with
 $\m^{i+1} E = <x p>$.
\end{proposition}
\begin{proof}
 Let $\m = (y_1,\ldots, y_s)$.
 If $\m^{i+1} E = 0$ then set $x = 0$. We now consider the case when $\m^{i+1} E \neq 0$.
 We note that $\m^{i+1} E$ is generated by $y_1p, \ldots, y_sp$. As a generating set of $\m^{i+1} E$
 can be shortened to a minimal basis we get that $\m^{i+1} E = <y_j p>$  for some $j$. Set $x = y_j$.
 \end{proof}

\section{Proof of Theorem \ref{main-1}}
In this section we give
\begin{proof}[Proof of Theorem \ref{main-1}]
We may assume that $A$ is complete.
 Let $c = \low(M) = \lambda(M)$.

 \textit{Step-1:}   $\mu(\m^iM) = 1$ for $i = 0,\ldots,c-1$. In particular $M$ is cyclic.\\
 Consider $G(M) = \bigoplus_{i = 0}^{c-1}\m^i M/\m^{i+1} M$. Observe that $G(M)$ has finite length as an $A$-module and
 $\lambda(G(M)) = \lambda(M) = c$. The result follows.

 Let $M = A/I$ for some ideal $I$.

 \textit{Step-2:} $A$ is a Gorenstein ring and $I$ is a Gorenstein ideal.\\
 Let $M^\vee$ be the Matlis dual of $M$. Note
 $$\low(M^v) = \low(M) = \lambda(M) = \lambda(M^\vee).$$
 So by Step 1 we get that $M^\vee$ is also cyclic. Note injective dimension of $M^\vee$ is finite. By Peskine and Szpiro \cite[II, 5.5]{PS}   we get that $A$ is a Gorenstein ring.

 Set $B = M = A/I$. Let the maximal ideal of $B$ be $\n$. By Step 1 the number of generators of $\n$ is utmost one.
 So $\m = (a, I)$ for some $a \in \m$. We may also assume that $A$ is not regular.

 By induction on $d$ we prove that $A$ is a hypersurface ring.

 We first consider the case when $d = 1$. Note a free resolution of $A/I$ is
 \[
  0 \rt A \xrightarrow{\phi} A \rt A/I \rt 0.
 \]
Notice $\phi$ is multiplication by some element $t \in A$. So $I = (t)$. Thus $\m = (a, t)$.
As $A$ is not regular we get that $A$ is a hypersurface ring.

Now assume $d \geq 2$ and the result has been proved for rings of dimension $d-1$.

We now make the following:\\
Claim 1: $I \nsubseteq \m^2$.\\
Claim 2: $I \nsubseteq P$ for any $P \in \Ass(A)$

 We finish the proof by assuming Claims1,2 which we prove later.

 By prime avoidance there exists $x \in I$ such that $x$ is $A$-regular and $x \notin \m^2$. Note that $M$ is also a
 $C = A/(x)$-module. By a result of Nagata (cf., \cite[3.3.5]{LLAV})),
 \[
  P^C_M(z) = \frac{P^A_M(z)}{(1+z)}
 \]
By \ref{divide} we get that $P_M^C(z)$ is a polynomial. In particular $\projdim_C M$ is finite.
By our induction hypothesis it follows that $C$ is a hypersurface ring. As $C = A/(x)$ and $x\notin \m^2$ is $A$-regular it follows
that $A$ is also a hypersurface ring.

Proof of Claim 1: Suppose if possible $I \subseteq \m^2$. Then by Nakayama Lemma we get $\m = (a)$. In particular $d \leq 1$
which is a contradiction.

Proof of Claim 2: If $I \subseteq P$ for some $P \in \Ass(A)$ then $\dim A/I \geq \dim A/P \geq  2$, a contradiction.
\end{proof}

We now give
\begin{proof}[Proof of Proposition \ref{Ex-1}]
 By prime avoidance choose an $A$-regular sequence \\  $x_1,\ldots, x_d$ with $x_i \in \m \setminus \m^2$ for all $i = 1,\ldots, d$.
 Set $B = A/(x_1,\ldots,x_d)$. Then $B$ has finite projective dimension over $A$ and $B$ has finite length.
 Also note that $B$ is a hypersurface. As $B$ is complete we get that $B = Q/(f)$ for some DVR $Q$. It is now easy to verify
 that $\low(B) = \lambda(B)$. Set $M = B$.
\end{proof}

\begin{remark}
 In the proof of the above Proposition we do not make any claim on $\low(M)$. If residue field of $A$ is infinite then we can choose
 $x_1,\ldots, x_d$ to be a superficial sequence. Then
 $$\low(M) = \lambda(M) = e(A) = \ord(A). $$
 The assertion $\lambda(M) = e(A)$, the multiplicity of $A$, is well known among researches in Blow up algebra's, for instance see \cite[Corollary 10]{Pu1}.
\end{remark}

\section{Proof of Theorem \ref{main-2}}
In this section we give:
\begin{proof}[Proof of Theorem \ref{main-2}]
We  assume $A$ is complete.
 Let $c = \low(M)$. By hypotheses $\lambda(M) = c + 1$. Note  $M^\vee$ also satisfies the hypotheses of the theorem since $\lambda(M^\vee) - \low(M^\vee) = 1$, see \ref{dual}.

 We consider the associated graded  module
 $G(M) = \bigoplus_{i = 0}^{c-1}\m^i M/\m^{i+1}M$. We consider two cases:\\
 \textit{Case 1:} $\mu(M) = 2$. Then $\mu(\m^i M) = 1 $ for $i = 1,\ldots c-1$.\\
 \textit{Case 2:} $\mu(M) = 1$. Then $B = M = A/I$ for some ideal $I$ in $A$. We note that
 $G(B) = G(M)$ is a standard $k$-algebra.
 So if $\mu (\m B) = 1$ then $\mu(\m^i B) \leq 1$ for all $i \geq 1$.
 This forces $\lambda(M) = \lambda(G(M)) \leq c$ which is a
 contradiction. So we have $\mu(\m M) = 2$ and $M$ is cyclic.

 We first consider Case 2 when $A$ is not regular.\\
 We prove by induction on dimension $d$ that $\ed(A) - d \leq 2$ and $\type(A/I) = 1$. We first note that there exists $a,b \in \m$
 such that $\m = (a, b, I)$.

 We first consider the case when $d = 1$.
 The projective resolution of $A/I$ is
 \[
  0 \rt A^m \rt A \rt A/I \rt 0.
 \]
After localizing at a minimal prime we get $m = 1$.
We also get that $I$ is principal. Thus $\ed(A) - 1 \leq 2$. As $A$  is Gorenstein we in-fact get that $A$ is a complete intersection of codimension $\leq 2$.  We also get $A/I$ is a complete intersection. So $\type(A/I) = 1$.

We now consider the case $d \geq 2$.\\
Claim-1:  $I \nsubseteq \m^2$.\\
Otherwise by Nakayama Lemma we get that $\m = (a,b)$, This is not possible if $d \geq 3$. If $d = 2$ this implies
that $A$ is regular, again a contradiction.

Claim-2: $I \nsubseteq P$ for every $P \in \Ass(A)$. \\
This is clear since $\height P = 0$ and $\height I = d \geq 2$.

  By prime avoidance there exists $x \in I$ such that $x$ is $A$-regular and $x \notin \m^2$. Note that $M = A/I$ is also a
 $C = A/(x)$-module. By a result of Nagata (cf., \cite[3.3.5]{LLAV})),
 \[
  P^C_M(z) = \frac{P^A_M(z)}{(1+z)}
 \]
By \ref{divide} we get that $P_M^C(z)$ is a polynomial. In particular $\projdim_C M$ is finite.
By our induction hypothesis it follows that $\ed(C) - (d-1) \leq 2$  and $\type(A/I) = 1$. As $C = A/(x)$ and
$x\notin \m^2$ is $A$-regular it follows
that $\ed(A) - d \leq 2$. As $A$ is Gorenstein we get that $A$ is a complete intersection of codimension $\leq 2$.

We now consider Case 1:
Suppose if possible $A$ is NOT regular.
We assert that $\type(M) = 1$ is not possible. Otherwise $M^\vee$ is a cyclic module with $\lambda(M^\vee) = \low(M^\vee) + 1$
 and $\type(M^\vee) = 2$. This contradicts Case 2.

So $\type(M) \geq 2$. As $\lambda(M^\vee) = \low(M^\vee) + 1$ we get $\type(M) = 2$.
Clearly $\m^{c-1}M \subseteq \soc(M)$.  Suppose if possible there exists some $i$
 with $1\leq i \leq c-2$ there
exists $p \in \m^iM \setminus \m^{i+1}M$ such that $p \in \soc(M)$. Then $\m^i M = <p>$. So  by \ref{lean}   there exists $x \in \m$ with
$\m^{i+1}M  =  < xp >  =  0$, a contradiction.
 As $\type(M) = 1$ is not possible it follows that there exists atleast one minimal generator of $M$ which is in the socle of $M$.

 If both the generators of $M$ are in the socle of $M$ then $M$ is a vector space. So $A$ is regular.
 Otherwise $M = < u, v>$ where $u \in \soc(M)$ and $v \notin \soc(M)$. 
 Set $U = Au$ and $V = Av$. Then $U + V = M$. If $\theta \in U\cap V$ then $\theta = \alpha u = \beta v$.
 If $\alpha \notin \m$ then we get $u = \alpha^{-1}\beta v$, a contradiction. So $\alpha \in \m$. This implies $\theta = 0$.
 So $M = U \oplus V$. It follows that $\projdim U < \infty$. As $U \cong k$ we get that $A$ is regular. Thus our assumption that $A$ is not regular is wrong. So $A$ is regular.

 If $A$ is regular and $\type(M) \geq 2$ then the same argument as before forces $\type(M) = 2$ and  $k$ to be a direct summand of $M$.
 If $\type(M) = 1$ then $M^\vee$ satisfies Case (2) with $A$ regular.

Next we analyze case (2) when $A$ is regular.

We first assert that $\dim A = 1$ is not possible. Suppose if possible $(A,\m)$ is one dimensional regular local ring and $M$ is a cyclic $A$-module with $\lambda(M) = \low(M) + 1$. Note $\m = (\pi)$ and $M = A/(\pi^s)$ for some $s \geq 1$. Then note $\lambda(M) = \low(M) = s$, a contradiction. Thus $\dim A = 1$ is not possible.

Now assume $\dim A \geq 3$ and $M = A/I$ satisfies our hypothesis. Then as $\lambda(\m/ \m^2 + I) = 2$ it follows that there exists $x \in I \setminus \m^2$. Thus we may reduce dimension of $A$ and so we can reduce to the case when $\dim A = 2$.

We now consider the case when $\dim A = 2$ and $M = A/I$ satisfies our hypotheses. If there exists $x \in I \setminus \m^2$ then $M$ is a module over the DVR $A/(x)$ which is a contradiction.
So $I \subseteq \m^2$. We have an exact sequence
\[
0 \rt \frac{I + \m^3}{\m^3} \rt \frac{\m^2}{\m^3} \rt \frac{\m^2}{I + \m^3} = \m^2M/\m^3M \rt 0
\]
 Now $\lambda(\m^2 M) \leq 1$. If $\m^2 M = 0$ then $I = \m^2$. Note $\type M = 2$.

 If $\m^2 M \neq 0$ then note that $\lambda((I + \m^3)/\m^3) = 2$.

 As $M^\vee$ also satisfies our hypotheses it follows that $\type M \leq 2$.
 If $\type M = 1$ then $M = A/I$ is a Gorenstein ring. It is well known that in this case $I$ is generated by two elements. It follows that $I = (u, v)$ where the images of $u, y$ in $\m^2/\m^3$ is linearly independent.

Consider the case when $\type(M) = 2$ and $\m^2 M \neq 0$.
Clearly $\m^{c-1}M \subseteq \soc(M)$.  Suppose if possible there exists some $i$
 with $2\leq i \leq c-2$ there
exists $p \in \m^iM \setminus \m^{i+1}M$ such that $p \in \soc(M)$. Then $\m^i M = <p>$. So by \ref{lean} there exists $x \in \m$ with
$\m^{i+1}M  =  < xp >  =  0$, a contradiction.
 So  there exists atleast one minimal generator of $\m M$ which is in the socle of $M$. As $M = A/I$ and $I \subseteq \m^2$ we may assume that there is a minimal generator  $x$ of $\m$ whose image in $A/I$ is in the socle of $A/I$. Let $x,y$ be a minimal generating set of $\m$. Then $x^2, xy \in I$. As easy computation shows that in-fact $I = (x^2, xy, y^n)$ for some $n \geq 3$.

 Further analysis of Case 2 when $A$ is not regular.

 Notice we may assume $\dim A = 1$ as when $\dim A \geq 2$ there exists $x \in I \setminus \m^2$.

  We first consider the case when $A$ is a hypersurface. Say $A = Q/(f)$ with $(Q,\n)$ regular local of dimension two and  $f \in \n^2$. By our proof $M = A/(x)$ for some non-zero divisor $x$ in $A$. Thus $M = Q/(f, x)$. As $\type M = 1$ it follows by our result when $Q$ is regular of dimension two and $M$ is cyclic with $\type(M) = 1$ that $f, x \in \n^2\setminus \n^3$.

  Next we consider the case when $A$ is a complete intersection of codimension two. Say $A = Q/(f, g)$ with $(Q,\n)$ regular local of dimension three and  $f,g \in \n^2$. By our proof $M = A/(x)$ for some non-zero divisor $x$ of $A$.

  Claim: $x \in \n \setminus \n^2$.\\
  Suppose if possible $x \in \n^2$.  Set $J = (f, g, x)$. We have an exact sequence
\[
0 \rt \frac{J + \n^3}{\n^3} \rt \frac{\n^2}{\n^3} \rt \frac{\n^2}{J + \n^3} = \n^2M/\n^3M \rt 0
\]
This yields $\lambda(\n^2 M/ \n^3 M) \geq 3$, a contradiction.

Set $\ov{Q} = Q/(x)$ a regular local ring of dimension two. Then $M = \ov{Q}/(\ov{f}, \ov{g})$. From our analysis  when $\ov{Q}$ is regular  of dimension two and $M$ is cyclic with $\type(M) = 1$
that $\ov{f},\ov{g} \in \n^2\setminus \n^3$.   Furthermore the images of $\ov{f}, \ov{g}$ in $\ov{\n}^2/\ov{\n}^3$ is linearly independent over $k = A/\m$. The result follows.
 \end{proof}

\s\label{examples-thm-2} \textbf{Examples:} In this subsection we give examples which shows that there exists local rings $A$ with $\ed(A) - \dim A \leq 2$ and an $A$-module $M$ with $\lambda(M) = \low(M) + 1$.

(i) If $A$ is regular we consider the following four examples:\\
(a) $A = k[[X, Y]]$. Set $M = A/(X^2, Y^2)$. Then $M$ is cyclic and $\lambda(M) = 4$ while $\low(M) = 3$.\\
(b) $A = k[[X, Y]]$. Set $U = A/(X^2, Y^2, XY)$ and $V = U^\vee$. Then both $U,V$ satisfy our constraints. Note $V$ is not cyclic.\\
(c) $M = k \oplus k$ does the job (here $M$ is NOT cyclic). \\
(d) $A = k[[X,Y]]$. Fix $n \geq 3$. Set
\[
M = \frac{A}{(X^2, XY, Y^n)} = k \oplus (kX \oplus kY) \oplus kY^2\oplus \cdots \oplus kY^{n-1}.
\]
Then note $\lambda(M) = n + 1$ while $\low(M) = n$.

(ii) $A = k[[X, Y]]/(Y^2)$. Then $\ed(A) - \dim A = 1$.  Set $M = A/(X^2)$. It is readily checked that $\lambda(M) = 4$ while $\low(M) = 3$.

(iii) $A = k[[X,Y,Z]]/(Y^2, Z^2)$.  Then $\ed(A) - \dim A = 2$.  Set $N = A/(X)$. It is readily checked that $\lambda(N) = 4$ while $\low(N) = 3$.


\begin{thebibliography}{10}

\bibitem{LLAV}
L.~L. Avramov,
\emph{Infinite free resolutions}, Six lectures on commutative algebra ({B}ellaterra, 1996), Progr.
  Math., vol. 166, Birkh\"auser, Basel, 1998, pp.~293--344.


\bibitem{ABIM}
L.~Avramov,  R-O.~Buchweitz,  S.~B.~Iyengar, and C.~Miller,
\emph{Homology of perfect complexes},
Adv. Math. 223 (2010), no. 5, 1731–-1781.

\bibitem{BH}
W.~Bruns and J.~Herzog, \emph{{Cohen-Macaulay rings}}, vol.~39, Cambridge
  studies in advanced mathematics, Cambridge University Press,~Cambridge, 1993.

  \bibitem{PS}
C.~Peskine and L.~Szpiro,
\emph{Dimension projective finie et cohomologie locale. Applications à la démonstration de conjectures de M. Auslander, H. Bass et A. Grothendieck}(French),
Inst. Hautes Études Sci. Publ. Math. No. 42 (1973), 47–-119.


\bibitem{Pu1}
T.~J. Puthenpurakal, \emph{{Hilbert coeffecients of a Cohen-Macaulay module}},
  J.~Algebra \textbf{264} (2003), 82--97.


\bibitem{P}
\bysame,
\emph{On the Loewy length of modules of finite projective dimension},
J. Commut. Algebra 9 (2017), no. 2, 291--301.

\end{thebibliography}
\end{document}